\documentclass[11 pt]{article}
\usepackage{amssymb}
\usepackage{amsmath}
\usepackage{amsthm}
\usepackage{amscd}
\usepackage{graphicx}
\usepackage{hyperref}

\renewcommand{\theequation}{\arabic{section}.\arabic{equation}}
\def\vbar{\mathchoice{\vrule height6.3ptdepth-.5ptwidth.8pt\kern- .8pt}
{\vrule height6.3ptdepth-.5ptwidth.8pt\kern-.8pt} {\vrule
height4.1ptdepth-.35ptwidth.6pt\kern-.6pt} {\vrule
height3.1ptdepth-.25ptwidth.5pt\kern-.5pt}}
\def\<{\langle}
\def\>{\rangle}
\def\a{\alpha}
\def\b{\beta}

\def\f{\phi}
\def\p{\psi}

\def\sw{\swarrow}
\def\nw{\nwarrow}
\def\ne{\nearrow}
\def\se{\searrow}
\def\pt{\prec}
\def\gr{\succ}

\newtheorem{df}{Definition}[section]
\newtheorem{thm}{Theorem}[section]
\newtheorem{cor}{Corollary}[section]
\newtheorem{rem}{Remark}[section]
\newtheorem{rems}{Remarks}[section]
\newtheorem{prop}{Proposition}[section]

\newtheorem{lem}{Lemma}[section]

\date{}
\begin{document}

\title{ BiHom-pre-alternative algebras and BiHom-alternative quadri-algebras}
\author{T. Chtioui, S. Mabrouk, A. Makhlouf}
\author{{ Taoufik Chtioui$^{1}$, Sami Mabrouk$^{2}$, Abdenacer Makhlouf$^{3}$ }\\
{\small 1.  University of Sfax, Faculty of Sciences Sfax,  BP
1171, 3038 Sfax, Tunisia} \\
{\small 2.  University of Gafsa, Faculty of Sciences Gafsa, 2112 Gafsa, Tunisia}\\
{\small 3.~ IRIMAS - D\'epartement de Math\'ematiques, 6, rue des fr\`eres Lumi\`ere,
F-68093 Mulhouse, France}}
 \maketitle
\begin{abstract}
The purpose of this paper is to introduce and study  the notion of BiHom-pre-alternative algebra which may be viewed as a BiHom-alternative algebra whose product can be decomposed into two compatible pieces.  Furthermore, we introduce the notion of BiHom-alternative quadri-algebra and show the connections between all these algebraic structures using Rota-Baxter operators and $\mathcal{O}$-operators.
\end{abstract}

 \normalsize\vskip0.5 cm

%----------------------------------------------------------------------------------------------------------
\section*{Introduction}
\renewcommand{\theequation}{\thesection.\arabic{equation}}

The  study of  periodicity phenomena in algebraic K-theory led
J.-L. Loday in 1990th to  introduce the notion of dendriform algebra (\cite{L1}) as the (Koszul) dual of the associative dialgebra. There is a remarkable fact that a Rota-Baxter operator (of weight zero), which first arose in
probability theory (\cite{Bax}) and later became a subject in combinatorics (\cite{R}), on an associative algebra naturally gives a dendriform algebra structure on the underlying vector space of the associative algebra (\cite{Ag1, Ag2, E1,L1,L3}). Such unexpected relationships between dendriform algebras in the field of operads and algebraic topology and Rota-Baxter operators in the field of combinatorics and probability  have been attracting a great interest because of their connections with various fields in mathematics and physics (see \cite{AbdaouiMabroukMakhlouf,EG,EMP,Gub} and the references therein). 
Later,  Aguiar and Loday introduced the notion of quadri-algebra (\cite{AL} in order to determine the algebraic structures behind a pair of commuting Rota-Baxter operators (on an associative algebra). This type of algebras are related  for example to  the space of  linear endomorphisms of an infinitesimal bialgebra and  have deep relationships with combinatorics and the theory of
Hopf algebras (\cite{AL}). A quadri-algebra is also regarded as the underlying algebra structure of a dendriform algebra with a nondegenerate 2-cocycle.
The connection between associative algebras and dendriform algebras were  generalized to  alternative
algebras by  X. Ni and  C. Bai (\cite{Bai3}) who introduced the  notion of pre-alternative algebra or pre-alternative-dendriform  algebra, as a generalization of  dendriform algebras. Analogs of quadri-algebras for alternative algebras had been obtained  by S. Madariaga, using the technique of splitting of operations (\cite{Bai4}, \cite{SaraMadariaga}).
Twisted algebras which were motivated by $q$-deformations of algebras of vector fields and  also  called Hom-algebras or BiHom-algebras have been intensively studied this last decade \cite{ChtiouiMabroukMakhlouf,GrazianiMakhloufMeniniPanaite,LiuMakhloufMeninPanaite,mak}.  In  (\cite{QinxiuSun}), Q. Sun  introduced the notion of Hom-pre-alternative algebra as a generalization of pre-alternative and Hom-dendriform algebras.

We aim in this  paper to introduce the BiHom version of  pre-alternative and  alternative quadri-algebras which generalize the classical structures.
We also introduce the notion of $\mathcal{O}$-operators of BiHom-alternative and BiHom-pre-alternative algebras. We will prove that given a BiHom-alternative algebra and  an $\mathcal{O}$-operator give rise to a BiHom pre-alternative algebra. In the same way we can construct  BiHom-alternative quadri-algebras  starting by a BiHom pre-alternative algebra.

This paper is organized as follows. In Section 1,  we summarize the
definitions of    BiHom-alternative, BiHom-pre-alternative algebra and their bimodules. We exploit the notion of $\mathcal{O}$-operator to  illustrate the relations existing between these structures.  In Section 2, we define   BiHom-alternative algebras and  BiHom-quadri-algebras. Therefore, we discuss their relationships    using $\mathcal{O}$-operators.

Throughout this paper $\mathbb{K}$ is a field of characteristic $0$ and all vector spaces are over $\mathbb{K}$. 
We refer to a  BiHom-algebra  as quadruple $(A,\mu,\alpha,\beta)$ where $A$ is a vector space, $\mu$ is a multiplication and $\alpha,\beta$ are two linear maps. It is said to be regular if $\alpha,\beta$ are invertible. A BiHom-associator is a trilinear map $as_{\alpha,\beta}$ defined for all $x,y,z\in A$ by 
$as_{\alpha,\beta}(x,y,z)=\alpha(x) (yz)-(xy)\beta(z)$.
When there is no ambiguity, we denote for simplicity the multiplication and composition by concatenation.

%%%%%%%%%%%%%%%%%%%%%%%%%%%%%%%%%%%%%%%
\section{BiHom-pre-alternative algebras}
%%%%%%%%%%%%%%%%%%%%%%%%%%%%%%%%%%%%%%%
In this section, we recall the notion of BiHom-alternative algebras given in \cite{ChtiouiMabroukMakhlouf} and  we introduce BiHom-pre-alternative algebras which are a generalization of Hom-type structures given in \cite{QinxiuSun}.
Moreover we provide some key constructions.

\begin{df}
A left BiHom-alternative algebra $($resp. right BiHom-alternative algebra) is a quadruple $( A,\mu,\alpha,\beta)$ consisting of  a $\mathbb{K}$-vector space $ A$, a  bilinear map $\mu: A \times  A \longrightarrow  A$ and two homomorphisms $\alpha,\beta:  A \longrightarrow A$ such that $\alpha \beta=\beta\alpha$, $\alpha\mu=\mu\alpha^{\otimes^2}$ and $\beta\mu=\mu\beta^{\otimes^2}$ that satisfy the left BiHom-alternative identity, i.e. for all $x,y,z \in  A$, one has
\begin{equation}\label{sa1}
as_{\alpha,\beta}(\beta(x),\alpha(y),z)+
as_{\alpha,\beta}(\beta(y),\alpha(x),z)=0,
\end{equation}
respectively, the right BiHom-alternative identity, i.e. for all $x,y,z \in  A$, one has
\begin{equation}\label{sa2}
as_{\alpha,\beta}(x,\beta(y),\alpha(z))+
as_{\alpha,\beta}(x,\beta(z),\alpha(y))=0.
\end{equation}
A BiHom-alternative algebra is one which is both a left and a right BiHom-alternative algebra.
\end{df}

\begin{df}
Let $( A,\mu,\a,\b)$ be a BiHom-alternative algebra and $V$ be a vector space. Let
$L,R:  A \to gl(V)$ and  $\phi,\psi \in gl(V)$  two commuting linear maps. Then $(V,L,R,\phi,\psi)$ is
called a bimodule, or a representation, of $( A,\mu,\a,\b)$, if for any $x,y \in A,\ v \in V$,
\begin{align}
&\phi L(x)=L(\a(x))\phi,\  \phi R(x)=R(\a(x))\phi,\\
& \psi L(x)=L(\b(x))\psi,\  \psi R(x)=R(\b(x))\psi \\
& L(\b(x)\a(x))\psi(v)  = L(\a\b(x))L(\a(x))v ,  \label{rep1}\\
& R(\b(x)\a(x))\phi(v)=R(\a\b(x))R(\b(x))v, \label{rep2}\\
&R(\b(y))L(\b(x))\phi(v)-L(\a\b(x))R(y)\phi(v) = R(\a(x)y)\phi\psi(v)-R(\b(y))R(\a(x))\psi(v) ,\label{rep3}\\
&L(\a(y))R(\a(x))\psi(v)-R(\a\b(x))L(y)\psi(v)=L(y\b(x))\phi\psi(v)-L(\a(y))L(\b(x))\phi(v).\label{rep4}
\end{align}
\end{df}
\begin{prop}A tuple 
$(V,L,R,\phi,\psi)$ is a bimodule of a BiHom-alternative algebra $( A,\mu,\a,\b)$ if and only if the direct sum $( A\oplus V,\ast,\a+\phi,\b+\psi)$  is turned into a BiHom-alternative algebra (the semidirect product) where
\begin{align}
   & (x_1+v_1)\ast (x_2+v_2)=\mu(x_1,x_2)+L(x_1)v_2+R(x_2)v_1, \label{product}\\
   & (\a+\phi)(x+v)=\a(x)+\phi(v),\ \ (\b+\psi)(x+v)=\b(x)+\psi(v).
    \end{align}
\end{prop}
\begin{proof}
For  any $x,y \in A$ and $u,v \in V$, we have
\begin{eqnarray*}
  &&as_{\a+\phi,\b+\psi}\big((\b+\psi)(x+u),(\a+\phi)(x+u),y+v\big)   \\
   &=&\Big((\b(x)+\psi(u))\ast(\a(x)+\phi(u))\Big)\ast\Big(\b(y)+\psi(v)\Big)  \\
   &-&\Big(\a\b(x)+\phi \psi(u)\Big)\ast\Big((\a(x)+\phi(u))\ast(y+v)\Big)\\
   &=&\Big(\b(x)\a(x)+L(\b(x))\phi(u)+R(\a(x))\psi(u) \Big)\ast\Big(\b(y)+\psi(v)\Big) \\
   &-& \Big(\a\b(x)+\phi \psi(u)\Big)\ast\Big(\a(x)y+L(\a(x))v+R(y)\phi(u)\Big) \\
   &=&(\b(x)\a(x))\b(y)+L(\b(x)\a(x))\psi(v)+R(\b(y))L(\b(x))\phi(u)+R(\b(y))R(\a(x))\psi(u)  \\
   &&-\a\b(x)(\a(x)y)-L(\a\b(x))L(\a(x))v-L(\a\b(x))R(y)\phi(u)-R(\a(x)y)\phi \psi(u).
\end{eqnarray*}
Hence  $as_{\a+\phi,\b+\psi}\big((\b+\psi)(x+u),(\a+\phi)(x+u),y+v\big) =0$ if and only if  Eqs. \eqref{rep1} and \eqref{rep3} hold. Analogously, $as_{\a+\phi,\b+\psi}\big(y+v,(\b+\psi)(x+u),(\a+\phi)(x+u)\big) =0$
if and only if \eqref{rep2} and \eqref{rep4} hold.\\ On the other hand, $(\a+\phi)(\b+\psi)=(\b+\psi)(\a+\phi)$, since $\phi \psi=\psi\phi$ and $\a\b=\b\a$.
Finally, the multiplicativity of $\a+\phi$ and $\b+\psi$ follow from the facts that
$\phi \psi=\psi\phi,\ \phi L(x)=L(\a(x))\phi,\ \phi R(x)=R(\a(x))\phi,\ \psi L(x)=L(\b(x))\psi$ and $\psi R(x)=R(\b(x))\psi$.
\end{proof}The following result gives a construction of a bimodule of  a BiHom-alternative algebra starting with a classical one by means of the Yau twist procedure.
\begin{prop}
Let $(V,L,R)$ be a bimodule of an alternative algebra $( A,\mu)$.  Given four linear maps $\a,\b: A \to  A$ and $\phi,\psi:V \to V$ such that $\a\b=\b\a,\ \phi \psi=\psi\phi,\ \phi L(x)=L(\a(x))\phi,\ \phi R(x)=R(\a(x))\phi,\ \psi L(x)=L(\b(x))\psi,$ and $\psi R(x)=R(\b(x))\psi$. Then  $(V,\widetilde{L},\widetilde{R},\phi,\psi)$ is a bimodule of the BiHom-alternative algebra $( A,\mu_{\a,\b},\a,\b)$, where $\widetilde{L}(x)=L(\a(x))\psi$,  $\widetilde{R}(x)=R(\b(x))\phi$ and $\mu_{\a,\b}(x,y)=\mu(\alpha(x),\beta( y)$.

\end{prop}
\begin{proof}
Let $x,y\in A$ and $v\in V$ and set  $\mu(x,y)=xy$, $\mu_{\a,\b}(x,y)=x\ast y$.   Then
\begin{align*}
&\widetilde{L}(\b(x)\ast\a(x))\psi(v)  - \widetilde{L}(\a\b(x))\widetilde{L}(\a(x))v \\&=L(\a^2\b(x)\a^2\b(x))\psi^2(v)  -L(\a^2\b(x))L(\a^2\b(x))\psi^2(v)
=0,
\end{align*}
and \begin{align*}
&\widetilde{R}(\b(y))\widetilde{L}(\b(x))\phi(v)-\widetilde{L}(\a\b(x))\widetilde{R}(y)\phi(v) - \widetilde{R}(\a(x)\ast y)\phi \psi(v)+\widetilde{R}(\b(y))\widetilde{R}(\a(x))\psi(v) \\&=R(\b^2(y))L(\a^2\b(x))\phi^2\psi(v)-L(\a^2\b(x))R(\b^2(y))\phi^2\psi(v)\\& - R(\a^2\b(x)\b^2(y))\phi^2\psi(v)+R(\b^2(y))R(\a^2\b(x))\phi^2\psi(v)
=0.
\end{align*}  The other identities can be shown using similar computations.

\end{proof}

Let $(V,L,R,\phi,\psi)$ be a bimodule  of a BiHom-alternative algebra $( A,\mu,\a,\b)$ and let $L^*,R^*:  A \to gl(V^*), \a^*,\b^*: A^*\to  A^*, \phi^*,\psi^*: V^* \to V^*$ be the dual maps of respectively $\a,\b,\phi$ and $\psi$ given by
\begin{align}\label{dual}
    & <L^*(x)u^*,v>=<u^*,L(x)v>, \hspace{0.3 cm} <R^*(x)u^*,v>=<u^*,R(x)v> \\
    & \a^*(x^*)(y)=x^*(\a(y)), \hspace{0.5 cm}  \b^*(x^*)(y)=x^*(\b(y)) \\
    &  \phi^*(u^*)(v)=u^*(\phi(v)) , \hspace{0.5 cm}       \psi^*(u^*)(v)=u^*(\psi(v))
\end{align}
 \begin{prop}
Let $(V,L,R,\phi,\psi)$ be a bimodule  of a BiHom-alternative algebra $( A,\mu,\a,\b)$.
Then $(V^*,L^*,R^*,\phi^*,\psi^*)$ is a bimodule of $( A,\mu,\a,\b)$ provided that
\begin{align}
   & \psi(L(\b(x)\a(x)))u=L(\a(x))L(\a\b(x))u ,\\
   & \phi(R(\b(x)\a(x)))u=R(\b(x))R(\a\b(x))u,\\
& \phi L(\b(x))R(\b(y))u-\phi R(y)L(\a\b(x))u=\psi\phi R(\a(x)y)u-\psi R(\a(x))R(\b(y))u,\\
& \psi R(\a(x))L(\a(y))u-\psi L(y)R(\a\b(x))u=\psi\phi L(y\b(x))u-\phi L(\b(x))L(\a(y))u,
\end{align}
for all $x,y \in  A$ and $u \in V$.
\end{prop}

\begin{proof}
Straightforward.
\end{proof}

\begin{df}\cite{Bai3}
A pre-alternative algebra is a triple $(A,\prec,\succ)$, where $A$ is a vector space and $\prec,\succ: A\times A\rightarrow A$ are two bilinear maps, satisfying
\begin{eqnarray}
 as^r(x,y,y)=0,\; as^l(x,x,y)=0,\\
as^m(x,y,z)+as^r(y,x,z)=0,\\
as^m(x,y,z)+as^l(x,z,y)=0,
\end{eqnarray}
where
\begin{align}\label{ass}
& as^r(x,y,z)=(x\pt y)\pt z-x \pt (y \pt z + y \gr z)\ \ \textrm{(right-associator)}, \\
& as^m(x,y,z)= (x \gr y) \pt z -x \gr (y \pt z)\  \ \textrm{(middle-associator)}, \\
& as^l(x,y,z)= (x \pt y+ x \gr y) \gr z- x \gr (y \gr z)\ \ \textrm{(left-associator)}.
\end{align}
\end{df}
\begin{prop}\cite{Bai3}
Let $(A, \pt, \gr )$ be a pre-alternative algebra. Then the operation
$$x \circ y= x \pt y + x \gr y, \ \  \textrm{for all}\ x,y \in A,$$
defines an alternative algebra, which is called the associated alternative algebra of
A and denoted by $Alt(A)$. We call $(A, \pt, \gr )$ a compatible pre-alternative algebra
structure on the alternative algebra $Alt(A)$.
\end{prop}

Now we give the BiHom version of a pre-alternative algebra.
\begin{df}
A BiHom-pre-alternative algebra (which may be called also BiHom-alternative-dendriform dialgebra) is a $5$-tuple $(A,\prec,\succ,\alpha,\beta)$  consisting of a vector  space $A$,  bilinear maps $\prec,\succ:A \times A \rightarrow A$ and two commuting  linear maps $\a,\b: A \rightarrow A$ satisfying the following conditions (for all $x,y,z \in A$):
\begin{eqnarray}
  \alpha(x\prec y)=\alpha(x)\prec \alpha(y), \; \ \alpha(x \succ y)=\alpha(x)\succ \alpha(y),\\
  \beta(x\prec y)=\beta(x)\prec \beta(y),\; \ \beta(x \succ y)=\beta(x)\succ \beta(y),\\
  as^r_{\alpha,\beta}(x,\beta(y),\alpha(z))+
      as^r_{\alpha,\beta}(x,\beta(z),\alpha(y))=0\label{right cond},\\
as^l_{\alpha,\beta}(\beta(x),\alpha(y),z)+
 as^l_{\alpha,\beta}(\beta(y),\alpha(x),z)=0\label{left cond},\\
 as^m_{\alpha,\beta}(\beta(x),\alpha(y),z)+
      as^r_{\alpha,\beta}(\beta(y),\alpha(x),z)=0\label{m,r cond},\\
   as^m_{\alpha,\beta}(x,\beta(y),\alpha(z))+
      as^l_{\alpha,\beta}(x,\beta(z),\alpha(y))=0\label{m,l cond},
\end{eqnarray}
where
\begin{align}
&   as_{\alpha,\beta}^r (x,y,z)=(x\prec y)\prec \beta(z)-\alpha(x) \prec(y\circ z), \\
 &  as_{\a,\b}^m (x,y,z)=(x\succ y)\prec \b(z)-\a(x)\succ (y \prec z) ,\\
 &   as_{\alpha,\beta}^l (x,y,z)=(x \circ y) \succ \b(z)-\a(x) \succ (y \succ z),
\end{align}
named respectively,   right-BiHom-associator, middle-BiHom-associator and left-BiHom-associator. \\
We call $\alpha $ and $\beta $ (in this order) the structure maps
of $A$.\\Note that
$x \circ y =x \prec y + x \succ y$.
\end{df}
It is obvious to see that a BiHom-dendriform algebra is BiHom-pre-alternative, exactly as any BiHom-associative algebra is BiHom-alternative.

\begin{rem}
Since the characteristic of $\mathbb{K}$ is $0$, conditions \eqref{right cond} and \eqref{left cond} are equivalent, respectively to 
\begin{eqnarray}
  as^r_{\alpha,\beta}(x,\beta(y),\alpha(y))   &=& 0,\\
 as^l_{\alpha,\beta}(\beta(x),\alpha(x),y)  &=& 0.
\end{eqnarray}
\end{rem}

A morphism $f:(A, \prec , \succ , \alpha , \beta )\rightarrow (A', \prec ', \succ ', \alpha ', \beta ')$ of
BiHom-pre-alternative algebras is a linear map
$f:A\rightarrow A'$ satisfying $f(x\prec y)=f(x)\prec ' f(y)$,  $f(x\succ y)=f(x)\succ ' f(y)$, for all $x, y\in A$,
as well as $f\circ \alpha =\alpha '\circ f$ and $f\circ \beta =\beta '\circ f$.

\begin{thm} \label{Yaudend}
Let $(A, \prec , \succ )$ be a pre-alternative algebra and $\alpha , \beta :A\rightarrow A$ two
commuting pre-alternative algebra morphisms. Define $\prec _{(\alpha , \beta )}, \succ _{(\alpha , \beta )}:
A \times  A\rightarrow A$ by
$$
x\prec _{(\alpha , \beta )}y=\alpha (x)\prec \beta (y)\ \ \text{and}\ \
x\succ _{(\alpha , \beta )}y=\alpha (x)\succ \beta (y),
$$
for all $x, y\in A$. Then $A_{(\alpha , \beta )}:=(A, \prec _{(\alpha , \beta )}, \succ _{(\alpha , \beta )},
\alpha , \beta )$ is a BiHom-pre-alternative algebra, called the Yau twist of $A$. \\ Moreover, assume that
$(A', \prec ', \succ ')$ is another pre-alternative algebra and $\alpha ', \beta ':A'\rightarrow A'$ are
two commuting pre-alternative algebra morphisms. Let $f:A\rightarrow A'$ be  a
pre-alternative algebra morphism satisfying $f\circ \alpha =\alpha '\circ f$ and $f\circ \beta =\beta '\circ f$. Then
$f:A_{(\alpha , \beta )}\rightarrow A'_{(\alpha ', \beta ')}$ is a  BiHom-pre-alternative algebra morphism.
\end{thm}
\begin{proof}
Let $x,y,z \in A$. Then we have
\begin{eqnarray*}
  as^l_{\a,\b}l(\b(x),\a(x),y) &=& (\b(x)\circ_{(\a,\b)}\a(x))\succ_{(\a,\b)}\b(y)-\a\b(x)\succ_{(\a,\b)}(\a(x)\succ_{(\a,\b)}y)  \\
   &=& (\a\b(x)\circ \a\b(x))\succ_{(\a,\b)}\b(y)-\a\b(x)\succ_{(\a,\b)}(\a^2(x)\succ \b(y)) \\
   &=& (\a^2\b(x)\circ\a^2\b(x))\succ \b^2(y)-\a^2\b(x)\succ(\a^2\b(x)\succ \b^2(y)) \\
   &=&as^l(\a^2\b(x),\a^2\b(x),\b^2(y))=0.
\end{eqnarray*}
Using a similar computation, one can check that $as^r_{\a,\b}(x,\b(y),\a(y))=0$.\\
On the other hand, we get
\begin{eqnarray*}
   && as^m_{\alpha,\beta}(\beta(x),\alpha(y),z)+
      as^r_{\alpha,\beta}(\beta(y),\alpha(x),z) \\
   &&= (\b(x)\succ_{(\a,\b)}\a(y))\prec_{(\a,\b)} \b(z)-\a\b(x)\succ_{(\a,\b)} (\a(y) \prec_{(\a,\b)} z) \\
   &&+ (\b(y)\prec_{(\a,\b)} \a(x))\prec_{(\a,\b)} \beta(z)-\a\b(y) \prec_{(\a,\b)}(\a(x)\circ_{(\a,\b)} z) \\
   &&= (\a\b(x)\succ \a\b(y))\prec_{(\a,\b)}\b(z)-\a\b(x)\succ_{(\a,\b)}(\a^2(y)\prec \b(z)) \\
   &&+(\a\b(y)\prec \a\b(x))\prec_{(\a,\b)}\b(z)-\a\b(y)\prec_{(\a,\b)}(\a^2(x)\circ \b(z))  \\
   &&= (\a^2\b(x)\succ \a^2\b(y))\prec\b^2(z)-\a^2\b(x)\succ(\a^2\b(y)\prec \b^2(z)) \\
   &&+(\a^2\b(y)\prec \a^2\b(x))\prec\b^2(z)-\a^2\b(y)\prec(\a^2\b(x)\circ \b^2(z))  \\
   &&=as^m(\a^2\b(x),\a^2\b(y),\b^2(z))+as^r(\a^2\b(y),\a^2\b(x),\b^2(z))=0.
\end{eqnarray*}
The rest  is left to the reader.
\end{proof}

\begin{rem}
More generally,
let $(A, \prec , \succ , \alpha , \beta )$ be a BiHom-pre-alternative algebra and $\tilde{\alpha }, \tilde{\beta }:
A\rightarrow A$ two  BiHom-pre-alternative algebra morphisms such that any two of the maps $\alpha, \beta ,
\tilde{\alpha }, \tilde{\beta }$ commute. Define new multiplications on $A$ by
\begin{eqnarray*}
&&x\prec 'y=\tilde{\alpha }(x)\prec \tilde{\beta }(y) \;\;\;\;\;and\;\;\;\;\;
x\succ 'y=\tilde{\alpha }(x)\succ \tilde{\beta }(y),
\end{eqnarray*}
for all $x, y\in A$. Then  $(A, \prec ', \succ ', \alpha  \tilde{\alpha }, \beta \tilde{\beta })$
is a BiHom-pre-alternative algebra.
\end{rem}

\begin{prop}
Let $(A,\prec,\succ,\a,\b)$ be a BiHom-pre-alternative algebra. Then $(A,\circ,\a,\b)$ is a BiHom-alternative algebra with the operation
$$x\circ y=x\prec y+x\succ y,$$
for any $x,y \in A$ . We say that $(A,\circ,\a,\b)$ is  the associated BiHom-alternative algebra of $(A,\prec,\succ,\a,\b)$ and   $(A,\prec,\succ,\a,\b)$ is called a compatible BiHom-alternative algebra structure on the BiHom-alternative algebra $(A,\circ,\a,\b)$.
\end{prop}
\begin{proof}
In fact, for any $x,y \in A$, we have
\begin{eqnarray*}
 && as_{\a,\b}(\b(x),\a(x),y) =(\b(x)\circ \a(x))\circ \b(y)-\a\b(x)\circ (\a(x)\circ y) \\
   &&= (\b(x)\circ \a(x))\succ \b(y)+(\b(x)\circ \a(x))\prec \b(y) \\
   &&- \a\b(x)\succ (\a(x)\circ y) -\a\b(x)\prec (\a(x)\circ y)\\
   &&= (\b(x)\circ \a(x))\succ \b(y)+ (\b(x)\succ \a(x))\prec \b(y) \\
   &&+(\b(x)\prec  \a(x))\prec \b(y) - \a\b(x)\succ (\a(x)\succ y)  \\
   &&-\a\b(x)\succ (\a(x)\prec y) -\a\b(x)\prec (\a(x)\circ y)\\
   &&=as^l_{\a,\b}(\b(x),\a(x),y)+as^m_{\a,\b}(\b(x),\a(x),y)+
   as^r_{\a,\b}(\b(x),\a(x),y)=0.
\end{eqnarray*}
Similarly, we show that $as_{\a,\b}(x,\b(y),\a(y))=0$.
\end{proof}
In the following we show connections with BiHom-Jordan algebras and BiHom-Malcev algebras. We refer for the definitions to \cite{ChtiouiMabroukMakhlouf}.
\begin{cor}
Let $(A,\prec,\succ,\a,\b)$ be a regular BiHom-pre-alternative algebra. Then $(A,\star,\a,\b)$ is a BiHom-Jordan algebra with the multiplication 
$$x\star y=x\prec y+x\succ y+ \a^{-1}\b(y)\prec \a\b^{-1}(x)+\a^{-1}\b(y)\succ \a\b^{-1}(x),$$
for any $x,y \in A$.
\end{cor}
\begin{cor}
Let $(A,\prec,\succ,\a,\b)$ be a regular BiHom-pre-alternative algebra. Then $(A,[-,-],\a,\b)$ is a BiHom-Malcev algebra with the multiplication
$$[x,y]=x\prec y+x\succ y- \a^{-1}\b(y)\prec \a\b^{-1}(x)-\a^{-1}\b(y)\succ \a\b^{-1}(x),$$
for any $x,y \in A$.
\end{cor}

\begin{prop}
Let $(A,\prec,\succ,\a,\b)$ be a BiHom-pre-alternative algebra. Then $(A,l_{\succ},r_{\prec},\a,\b)$ is  a bimodule of the associated BiHom-alternative algebra $(A,\circ,\a,\b)$, where $l_{\succ}$ and $r_{\prec}$ are
the left and right multiplication operators corresponding respectively to the two multiplications $\prec,\succ$.
\end{prop}
\begin{proof}
Let $x,y,z \in A$. Then $as^r_{\a,\b}(y,\b(x),\a(x))=0$, that is
$$(y\prec \b(x))\prec \a\b(x)=\a(y)\prec(\b(x)\circ \a(x)),$$
which means that
$$r_{\prec}(\a\b(x))r_{\prec}(\b(x))y=r_{\prec}(\b(x)\circ \a(x))\a(y).$$
Similarly, $as^l_{\a,\b}(\b(x),\a(x),y)=0$ is equivalent to
$$l_{\succ}(\b(x)\circ \a(x))\b(y)=l_{\succ}(\a\b(x))l_{\succ}(\a(x))y.$$
On the other hand,  we have $as^m_{\a,\b}(\b(x),\a(y),z)+as^r_{\a,\b}(\b(y),\a(x),z)=0$, that is
\begin{eqnarray*}
(\b(x)\succ\a(y))\prec \b(z)-\a\b(x)\succ (\a(y) \prec z) 
  = \a\b(y) \prec(\a(x)\circ  z)-(\b(y)\prec \a(x))\prec \beta(z).
\end{eqnarray*}
This means that
$$r_{\prec}(\b(z))l_{\succ}(\b(x))\a(y)-l_{\succ}(\a\b(x))r_{\prec}(z)\a(y)
=r_{\prec}(\a(x)\circ z)\a\b(y)-r_{\prec}(\b(z))r_{\prec}(\a(x))\b(y).$$
Finally, since $  as^m_{\a,\b}(x,\b(y),\a(z))+as^l_{\a,\b}(x,\b(z),\a(y))=0$ then
\begin{eqnarray*}
(x\succ \b(y))\prec\a\b(z)-\a(x)\succ(\b(y)\prec \a(z)) 
   = (x \circ \b(z))\succ \a\b(y)-\a(x)\succ(\b(z)\succ \a(y)).
\end{eqnarray*}
Hence
$$r_{\prec}(\a\b(z))l_{\succ}(x)\b(y)-l_{\succ}(\a(x))r_{\prec}(\a(z))\b(y)=
l_{\succ} (x \circ \b(z))\a\b(y)-l_{\succ}(\a(x))l_{\succ}(\b(z))\a(y).$$
\end{proof}

The following definition introduces the notion of bimodule  of BiHom-pre-alternative algebras.
\begin{df}
Let $( A\prec,\succ,,\a,\b)$ be a BiHom-pre-alternative algebra. A Bimodule of $ A$ is a vector space $V$ together with two commuting linear maps $\f, \p:V \to V$ and four linear maps $L_\succ,L_\prec,R_\succ,R_\prec:  A\to gl(V)$
satisfying the following set of identities
\begin{align}
&  \f L_\prec(x)=L_\prec(\a(x))\f, \p L_\prec(x)=L_\prec(\b(x))\p,\\
& \f R_\prec(x)=R_\prec(\a(x))\f, \p R_\prec(x)=R_\prec(\b(x))\p,
\f L_\succ(x)=L_\succ(\a(x))\f, \\
&\p L_\succ(x)=L_\succ(\b(x))\p, \f R_\succ(x)=R_\succ(\a(x))\f,
\p R_\succ(x)=R_\succ(\b(x))\p,\\
 &L_\succ(\b(x)\circ \a(x))\p=L_\succ(\a\b(x))L_\succ(\a(x)),  \\
& R_\succ(\b(y))(L_\circ(\b(x))\f+R_\circ(\a(x))\p)=L_\succ(\a\b(x))R_\succ(y)\f+R_\succ(\a(x)\succ y)\f\p,\\
 &R_\prec(\a\b(x))R_\prec(\b(x))=R_\prec(\b(x)\circ \a(x))\f ,\\
&L_\prec(\a(y))(L_\circ(\b(x))\f+R_\circ(\a(x))\p)=L_\prec(y\prec\b(x)) \f\p+
R_\prec(\a\b(x))L_\prec(y)\p ,\\
 &L_\prec(\b(x)\succ \a(y+\b(y)\prec \a(x))\p=L_\succ(\a\b(x))L_\prec(\a(y))+L_\prec(\a\b(y))L_0(\a(x)), \\
&  R_\prec(\b(y))(L_\succ(\b(x))\f+R_\prec(\a(x))\p)=
L_\succ(\a\b(x))R_\prec(y)\f+R_\prec(\a(x)\circ y)\f\p,\\
&R_\prec(\b(y))(R_\succ(\a(x))\p+L_\prec(\b(x))\f)=
R_\succ(\a(x)\prec y)\f\p+L_\prec(\a\b(x))R_\circ(y)\f, \\
& R_\prec(\a\b(y))R_\succ(\b(x))+R_\succ(\a\b(x))R_\circ(\b(y))=
R_\succ(\b(x)\prec\a(y))\f+R_\succ(\b(y)\succ\a(x))\f,\\
&R_\prec(\a\b(y))L_\succ(x)\p+L_\succ(x\circ \b(y))\f\p=
L_\succ(\a(x))(R_\prec(\a(y))\p+L_\succ(\b(y))\f),\\
&L_\prec(x\succ\b(y))\f\p+R_\succ(\a\b(y))L_\circ(x)\p=
L_\succ(\a(x))L_\prec(\b(y))\f+L_\succ(\a(x))R_\succ(\a(y))\p,
\end{align}
where $\circ=\prec+\succ$, $L_\circ=L_\prec+L_\succ$ and $R_\circ=R_\prec+R_\succ$.
\end{df}

\begin{prop}A tuple 
$(V,L_\succ,L_\prec,R_\succ,R_\prec,\f,\p)$ is a bimodule of a BiHom-pre-alternative algebra $( A,\prec,\succ,\a,\b)$  if and only if the direct sum $( A\oplus V, \ll,\gg,\a+\f,\b+\p)$ is a BiHom-pre-alternative algebra, where
\begin{align*}
 & (x+u)\ll (y+v) =x\prec y+L_\prec(x)v+R_\prec(y)u,\\
 &  (x+u)\gg (y+v)= x\succ y+ L_\succ(x)v+R_\succ(x)u,\\
 & \text{and}\ (\a+\f)(x+u)=\a(x)+\f(u),\ (\b+\p)(x+u)=\b(x)+\p(u),
\end{align*}
for any $x,y \in  A$ and $u,v \in V$.
\end{prop}

\begin{proof}
Straightforward and  left to the reader.
\end{proof}

\begin{df}
Let $(V,L,R,\phi,\psi)$ be a bimodule of a BiHom-alternative algebra $(A,\circ,\a,\b)$. A linear map $T: V \rightarrow A$ is called an $\mathcal{O}$-operator associated to $(V,L,R,\phi,\psi)$ if for all $u,v \in V$
\begin{equation}\label{O-operator}
    T(u)\circ T(v)=T\big(L(T(u))v+R(T(v))u\big),\ \ T\phi=\a T,\ \ T\psi=\b T.
\end{equation}
\end{df}

\begin{rem}
A Rota-Baxter operator of weight $0$ on a  BiHom-alternative algebra $(A,\circ,\a,\b)$ is just an $\mathcal{O}$-operator associated to the bimodule $(A,\ell,r,\a,\b)$, where $\ell$ and $r$ are
the left and right multiplication operators corresponding to the multiplication $\circ$.
\end{rem}

\begin{prop}\label{O-operator}
Let $T:V \rightarrow A$ be an $\mathcal{O}$-operator of a BiHom-alternative algebra $(A,\cdot,\a,\b)$ associated to a bimodule $(V,L,R,\phi,\psi)$. Then $(V,\prec,\succ,\phi,\psi)$ is a BiHom-pre-alternative algebra structure, where
$$u\prec v=R(T(v))u \ \textrm{and}\ u\succ v=L(T(u))v,$$
for all $u,v \in V$. Therefore $(V,\circ=\prec+\succ,\phi,\psi)$ is the associated BiHom-alternative algebra of this BiHom-pre-alternative algebra and $T$ is a BiHom-alternative algebra morphism.
Furthermore, $T(V)=\{T(v);\ v \in V\}\subset A$ is a BiHom-alternative subalgebra of $(A,\cdot,\a,\b)$ and $(T(V),\prec',\succ',\a,\b)$ is a BiHom-pre-alternative algebra given by
$$T(u)\prec' T(v) =T(u\prec v),\ \ \ \text{and}\ \ \ \ T(u)\succ' T(v)=T(u \succ v)$$
for all $u, v \in V$.

Moreover, the associated BiHom-alternative algebra $(T(V),\bullet=\prec'+\succ',\a,\b)$  is just a BiHom-alternative subalgebra structure of $(A,\cdot,\a,\b)$  and $T$ is a
BiHom-alternative algebra morphism.

\end{prop}

\begin{proof}
For any $u,v,w \in V$,
\begin{eqnarray*}
  as^l_{\f,\p}(\psi(u),\phi(u),v) &=& (\p(u)\circ \f(u))\succ \p(v)-\f\p(u)\succ(\f(u)\succ v)  \\
   &=& L(T(\p(u)\circ \f(u)))\p(v)-\f\p(u)\succ(L(T(\f(u)))v) \\
   &=& L(T(\p(u))\circ T(\f(u)))\p(v)-L(T(\f\p(u)))L(T(\f(u)))v  \\
   &=&  0.
\end{eqnarray*}
Furthermore,
\begin{eqnarray*}
 as^l_{\f,\p}(u,\p(v),\f(w))&=& (u\circ \p(v))\succ \f\p(w)-\f(u)\succ(\p(v)\succ \f(w)) \\
&=& L(T(u\circ \p(v)))\f\p(w)-L(T(\f(u)))L(T(\p(v)))\f(w) \\
   &=& L(T(u)\circ \b(T(v)))\f\p(w)-L(\a(T(u)))L(\b(T(v)))\f(w),
\end{eqnarray*}
and
\begin{eqnarray*}
as^m_{\f,\p}(u,\p(w),\f(v))&=&(u\succ \p(w))\prec\f\p(v)-\f(u)\succ(\p(w)\prec \f(v)) \\
   &=& R(T(\f\p(v)))L(T(u))\p(w)-L(T(\f(u)))R(T(\f(v)))\p(w) \\
   &=& R(\a\b(T(v)))L(T(u))\p(w)-L(\a(T(u)))R(\a(T(v)))\p(w).
\end{eqnarray*}
Hence
\begin{eqnarray*}
   && as^l_{\f,\p}(u,\p(v),\f(w))+as^m_{\f,\p}(u,\p(w),\f(v))  \\
   &&=L(T(u)\circ \b(T(v)))\f\p(w)-L(\a(T(u)))L(\b(T(v)))\f(w)   \\
   &&+ R(\a\b(T(v)))L(T(u))\p(w)-L(\a(T(u)))R(\a(T(v)))\p(w)  \\
   &&= 0.
\end{eqnarray*}
The other identities for $(V,\prec,\succ,\phi,\psi)$ being a BiHom-pre-alternative algebra can be verified similarly.
\end{proof}

\begin{cor} \label{BHRBzero}
Let $(A, \mu , \alpha , \beta )$ be a BiHom-alternative algebra and $R:A\rightarrow A$ be a Rota-Baxter operator
of weight 0 such that $R \alpha =\alpha  R$, $R \beta =\beta  R$. Define the multiplications
$\prec $ and $\succ $ on $A$ by
\begin{eqnarray*}
&&x\prec y=xR(y)\;\;\;\;\;and\;\;\;\;\;x\succ y=R(x)y,
\end{eqnarray*}
for all $x, y\in A$. Then $(A, \prec , \succ , \alpha , \beta )$ is a BiHom-pre-alternative algebra.
\end{cor}

\begin{proof}
The proof is left to the reader.
\end{proof}
Now, we introduce a  definition of $1$-BiHom-cocycle. 
\begin{df}
Let $(A,\cdot ,\a,\b)$ be a BiHom-alternative algebra and  $(V,L,R,\f,\p)$ be a bimodule of $A$.  A linear map $D: A \rightarrow V$ is said to be a $1$-BiHom-cocycle  of $(A,.,\a,\b)$ into
 $(V,L,R,\f,\p)$  if, for all $x,y \in A$,
 $$D(xy)=L(x)D(y)+R(y)D(x),\ \ \, \f D=D \a,\ \ \ \p D=D \b.$$
\end{df}

\begin{prop}
Let $(A,.,\a,\b)$ be a BiHom-alternative algebra.  Then the following statements are equivalent
\begin{itemize}
  \item [(i)]  There is a compatible BiHom-pre-alternative algebra $(A,\prec,\succ,\a,\b)$ structure on $(A,\cdot ,\a,\b)$.
  \item [(ii)]  There is an invertible $\mathcal{O}$-operator associated to a bimodule of $(A,\cdot ,\a,\b)$.
  \item [(iii)] There is a bijective $1$-BiHom-cocycle of $(A,\cdot ,\a,\b)$ into a bimodule.
\end{itemize}
\end{prop}

\begin{proof}
$(iii)\Rightarrow (ii)$: If $D$  is a bijective $1$-BiHom-cocycle of $(A,\cdot ,\a,\b)$ into a bimodule $(V,L,R,\f,\p)$, then $D^{-1}$ is an $\mathcal{O}$-operator associated to $(V,L,R,\f,\p)$. Indeed, it's clear that
$D^{-1}\f=\a D^{-1}$ and $D^{-1}\p=\b D^{-1}$. Take two elements $u,v \in V$, there is $x, y \in A$ such that $D(x)=u$ and $D(y)=v$. Since
$D(xy)=L(x)D(y)+R(y)D(x),$
then
$$D^{-1}(u)D^{-1}(v)=D^{-1}\big(L(D^{-1}(u))v+R(D^{-1}(v))u\big).$$
$(ii) \Rightarrow (i)$:  If $T: V \rightarrow A$ is an
invertible $\mathcal{O}$-operator associated to a bimodule $(V, L,R,\f,\p)$, then $T(V)=A$ and using Proposition \ref{O-operator}, there is a compatible pre-alternative algebra structure on $A$ given by:
\begin{equation}
x\prec y=T(R(y)T^{-1}(x)),\;\;  x\succ y=T(L(x)T^{-1}(y)),\;\;
\forall x,y\in A.
\end{equation}
$(i)\Rightarrow (iii)$:  If $(A,\prec,\succ,\a,\b)$ is a compatible
BiHom-pre-alternative algebra structure on $(A,\cdot,\a,\b)$, then it is obvious
that the identity map $ id$ is a bijective $1$-BiHom-cocycle of $A$
into the adjoint bimodule $(A,l_{\succ}, r_{\prec},\a,\b)$.

\end{proof}

%%%%%%%%%%%%%%%%%%%%%%%%%%%%%%%%%%%%%%%%%%

%%%%%%%%%%%%%%%%%%%%%%%%%%%%%%%%%%%%%%%%%%

%%%%%%%%%%%%%%%%%%%%%%%%%%%%%%%%%%%%%%%%%%%
\section{BiHom-alternative quadri-algebras}
%%%%%%%%%%%%%%%%%%%%%%%%%%%%%%%%%%%%%%%%%%%
In this section, we introduce the BiHom version of alternative quadri-algebras introduced in \cite{SaraMadariaga}. This variety of algebras is a generalization of BiHom-quadri-algebras discussed in \cite{LiuMakhloufMeninPanaite}.
\begin{df}
A BiHom-alternative quadri-algebra is a 7-tuple $(A, \nwarrow, \swarrow, \nearrow, \searrow, \alpha , \beta )$
consisting of a vector  space $A$, four bilinear maps $\nwarrow,
\swarrow, \nearrow, \searrow: A \times A\rightarrow A$ and two commuting linear maps $\alpha ,
\beta : A\rightarrow A$ which are algebra maps with respect to previous four operations  and such that   the following axioms hold for all $x, y, z\in A$
\begin{eqnarray}
%&&\alpha (x\nearrow y)=\alpha (x)\nearrow \alpha (y), ~~~~ \alpha (x\searrow y)=\alpha (x)\searrow \alpha (y), \label{BiHomqua5} \\
%&&\alpha (x\nwarrow y)=\alpha (x)\nwarrow \alpha (y), ~~~~ \alpha (x\swarrow y)=\alpha (x)\swarrow \alpha (y), \label{BiHomqua6} \\
%&&\beta (x\nearrow y)=\beta (x)\nearrow \beta (y),~~~~  \beta (x\searrow y)=\beta (x)\searrow \beta (y),\label{BiHomqua7} \\
%&&\beta (x\nwarrow y)=\beta (x)\nwarrow \beta (y),~~~~ \beta (x\swarrow y)=\beta (x)\swarrow \beta (y), \label{BiHomqua8} \\
&&  \{\b(x),\a(y),z\}^r_{\a,\b}+\{\b(y),\a(x),z\}_{\a,\b}^m=0 ,  \label{BiHomqua9} \\
&&   \{\b(x),\a(y),z\}^n_{\a,\b}+\{\b(y),\a(x),z\}_{\a,\b}^w=0 ,     \label{BiHomqua10} \\
&&    \{\b(x),\a(y),z\}^{ne}_{\a,\b}+\{\b(y),\a(x),z\}_{\a,\b}^e=0 ,         \label{BiHomqua11}\\
&&    \{\b(x),\a(y),z\}^{sw}_{\a,\b}+\{\b(y),\a(x),z\}_{\a,\b}^s=0 ,        \label{BiHomqua12}\\
&&    \{\b(x),\a(y),z\}^l_{\a,\b}+\{\b(y),\a(x),z\}_{\a,\b}^l=0,          \label{BiHomqua13}\\
&&    \{x,\b(y),\a(z)\}^r_{\a,\b}+\{x,\b(z),\a(y)\}_{\a,\b}^r=0 ,          \label{BiHomqua14}\\
&&    \{x,\b(y),\a(z)\}^n_{\a,\b}+\{x,\b(z),\a(y)\}_{\a,\b}^{ne}=0,           \label{BiHomqua15}\\
&&    \{x,\b(y),\a(z)\}^w_{\a,\b}+\{x,\b(z),\a(y)\}_{\a,\b}^{sw}=0,           \label{BiHomqua16}\\
&&    \{x,\b(y),\a(z)\}^m_{\a,\b}+\{x,\b(z),\a(y)\}_{\a,\b}^l=0 ,          \label{BiHomqua17}\\
&&    \{x,\b(y),\a(z)\}^s_{\a,\b}+\{x,\b(z),\a(y)\}_{\a,\b}^e=0;           \label{BiHomqua18}
\end{eqnarray}
where
\begin{align*}
& x\succ y:=x\nearrow y+ x\searrow y,    \hspace{1 cm}
 x\prec y:=x\nwarrow y+ x\swarrow y, 
  %\label{BiHomqua1} 
 \\
& x\vee y:=x\searrow y+ x\swarrow y,   \hspace{1 cm}
 x\wedge y:=x\nearrow y+ x\nwarrow y, 
 %\label{BiHomqua2} 
 \\
& x\ast y:=x\searrow y+ x\nearrow y+ x\swarrow y+ x\nwarrow y
 = x\succ y+ x\prec y= x\vee y+ x\wedge y.
 % \label{BiHomqua3}
\end{align*}
and 
\begin{align*}
    & \{x,y,z\}^r_{\a,\b}=(x\nwarrow y)\nwarrow \b(z)-\a(x)\nwarrow (y \ast z) ,\quad
     \{x,y,z\}^l_{\a,\b}=(x\ast y)\searrow \b(z)- \a(x)\searrow (y\searrow z)\\
    & \{x,y,z\}_{\a,\b}^{ne}=(x\wedge y)\nearrow \b(z)-\a(x)\nearrow (y\succ z), \quad
    \{x,y,z\}_{\a,\b}^{sw}=(x\prec y)\swarrow \b(z) -\a(x)\swarrow (y \vee z),\\
    &  \{x,y,z\}_{\a,\b}^n=(x\nearrow y)\nwarrow \b(z)-\a(x)\nearrow (y\prec z), \quad
      \{x,y,z\}_{\a,\b}^w=(x\swarrow y) \nwarrow \b(z)- \a(x)\swarrow (y \wedge z),\\
  &   \{x,y,z\}_{\a,\b}^s=(x \succ y) \swarrow \b(z)-\a(x)\searrow  (y\swarrow z), \quad
    \{x,y,z\}_{\a,\b}^e=(x \vee y)\nearrow \b(z)-\a(x)\searrow (y \nearrow z),\\
   &   \{x,y,z\}_{\a,\b}^m=(x\searrow y) \nwarrow \b(z)- \a(x) \searrow (y\nwarrow z).
\end{align*}

\end{df}

\begin{rems}$\bullet$
Note that in every BiHom-quadri-algebra, all these associators are trivial.
\\
%$\bullet$ These ten identities are independent (there are not redundancies).\\
$\bullet$
When $\a=\b=id$, the BiHom-alternative quadri-algebra is an alternative quadri-algebra.\\
$\bullet$
Since $\mathbb{K}$ is a field of characteristic different from $2$, the identities   \eqref{BiHomqua13} and   \eqref{BiHomqua14} can be written respectively 
$  \{\b(x),\a(x),y\}^l_{\a,\b}=0$  and $ \{x,\b(y),\a(y)\}_{\a,\b}^r=0.$

\end{rems}

\begin{lem}\label{quadri to pre}
 Let $(A, \nwarrow, \swarrow, \nearrow, \searrow, \alpha , \beta )$ be a  BiHom-alternative quadri-algebra. Then both $(A,\prec,\succ,\a,\b)$ and $(A,\vee,\wedge,\a,\b)$ are BiHom-pre-alternative algebras  (called, respectively, horizontal and vertical BiHom-pre-alternative
structures associated to $A$) and $(A,\ast,\a,\b)$ is a BiHom-alternative algebra.
\end{lem}
\begin{proof}
We will prove just  that  $(A,\prec,\succ,\a,\b)$  is a BiHom-pre-alternative algebra. For this, take $x, y, z \in A$.  We will just show how  to prove two identities and the other ones could  be done similarly. We remark that
\begin{align*}
& (x\prec\b(y))\prec \a \b (y)-\a(x)\prec (\b(y) \ast \a(y))\\
&= (x\nwarrow \b(y)) \nw \a \b(y)+(x \nw \b(y))\sw \a\b(y)+(x\sw \b(y))\nw \a\b(y) \\
&+(x \sw \b(y)) \sw \a\b(y)-\a(x) \nw (\b(y) \ast \a(y))-\a(x) \sw (\b(y) \vee \a(y))-
\a(x) \sw (\b(y) \wedge \a(y))  \\
&= \{x,\b(y),\a(y)\}^r_{\a,\b}+ \{x,\b(y),\a(y)\}^w_{\a,\b}+ \{x,\b(y),\a(y)\}^{sw}_{\a,\b}=0.
\end{align*}
In addition,
\begin{align*}
&as_{\a,\b}^m(\b(x),\a(y),z)+ as_{\a,\b}^r(\b(y),\a(x),z) \\
&=(\b(x) \gr \a(y)) \pt \b(z)-\a\b(x) \gr (\a(y) \pt z)+(\b(y) \pt \a(x)) \pt \b(z) -
\a\b(y) \pt ( \a(x) \ast z) \\
&=  (\b(x) \se \a(y)) \nw \b(z) +(\b(x) \se \a(y)) \sw \b(z) + (\b(x) \ne \a(y))\nw \b(z) \\
& + (\b(x) \ne \a(y)) \nw \b(z) - \a\b(x) \se (\a(y) \nw z)-\a\b(x) \se (\a(y) \sw z) \\
&- \a\b(x) \ne (\a(y) \nw z) -\a\b(x) \ne (\a(y) \sw z)  + (\b(y) \pt \a(x))\sw \b(z) \\
&+(\b(y) \nw \a(x)) \nw \b(z) + (\b(y) \sw \a(x)) \nw \b(z) -\a\b(y)\nw (\a(x) \ast z) \\
&-\a\b(y)\nw (\a(x) \vee z) -\a\b(y)\nw (\a(x) \wedge z) \\
&= \{\b(x),\a(y),z\}_{\a,\b}^m + \{\b(x),\a(y),z\}_{\a,\b}^n+  \{\b(x),\a(y),z\}_{\a,\b}^s \\
&+  \{\b(y),\a(x),z\}_{\a,\b}^r +  \{\b(y),\a(x),z\}_{\a,\b}^w + \{\b(y),\a(x),z\}_{\a,\b}^{sw}=0.
\end{align*}
Using a similar computation, we can get
\begin{align*}
&as_{\a,\b}^m(x,\b(y),\a(z))+ as_{\a,\b}^l(x,\b(z),\a(y)) \\
&=\{x,\b(y),\a(z)\}_{\a,\b}^m+\{x,\b(y),\a(z)\}_{\a,\b}^n+\{x,\b(y),\a(z)\}_{\a,\b}^s\\
&+\{x,\b(z),\a(y)\}_{\a,\b}^l+\{x,\b(z),\a(y)\}_{\a,\b}^{ne}+\{x,\b(z),\a(y)\}_{\a,\b}^e=0.
\end{align*}
Similarly,   we can show that  $(A,\vee,\wedge,\a,\b)$  and $(A, \ast,\a,\b)$ are respectively  BiHom-pre-alternative algebra and  BiHom-alternative algebra.
\end{proof}

A morphism $f:(A, \nwarrow, \swarrow, \nearrow, \searrow, \alpha ,
\beta )\rightarrow (A', \nwarrow', \swarrow', \nearrow', \searrow',
\alpha' , \beta' )$ of BiHom-alternative quadri-algebras is a linear map
$f:A\rightarrow A'$ satisfying $f(x\nearrow y)=f(x)\nearrow' f(y),
f(x\searrow y)=f(x)\searrow' f(y), f(x\nwarrow y)=f(x)\nwarrow'
f(y)$ and $f(x\swarrow y)=f(x)\swarrow ' f(y)$, for all $x, y\in A$,
as well as $f \alpha =\alpha ' f$ and $f \beta =\beta
' f$.

\begin{prop} \label{quad}
Let $(A, \nwarrow, \swarrow, \nearrow, \searrow )$ be an
alternative quadri-algebra and $\alpha , \beta :A\rightarrow A$ two commuting
alternative quadri-algebra morphisms. Define $\searrow _{(\alpha , \beta )},
\nearrow _{(\alpha , \beta )}, \swarrow _{(\alpha , \beta )},
\nwarrow _{(\alpha , \beta )}: A\times A \rightarrow A$ by
\begin{eqnarray*}
&&x\searrow _{(\alpha , \beta )}y=\alpha (x)\searrow \beta (y),
\quad \quad x\nearrow _{(\alpha , \beta )}y=\alpha (x)\nearrow \beta
(y),\\
&&x\swarrow _{(\alpha , \beta )}y=\alpha (x)\swarrow \beta (y),
\quad \quad x\nwarrow _{(\alpha , \beta )}y=\alpha (x)\nwarrow \beta
(y),
\end{eqnarray*}
for all $x, y\in A$.
Then $A_{(\alpha , \beta )}:=(A, \nwarrow _{(\alpha , \beta )},  \swarrow
_{(\alpha , \beta )}, \nearrow _{(\alpha , \beta )},
\searrow
_{(\alpha , \beta )},  \alpha , \beta )$ is a BiHom-alternative quadri-algebra, called the Yau twist of $A$. Moreover,
assume that $(A', \nwarrow', \swarrow', \nearrow', \searrow' )$ is
another alternative quadri-algebra and $\alpha ', \beta ':A'\rightarrow A'$ are
two commuting alternative quadri-algebra morphisms and $f:A\rightarrow A'$
is an alternative quadri-algebra morphism satisfying $f \alpha =\alpha
' f$ and $f \beta =\beta ' f$. Then $f:A_{(\alpha ,
\beta )}\rightarrow A'_{(\alpha ', \beta ')}$ is a
BiHom-alternative quadri-algebra morphism.
\end{prop}

 \begin{proof}
We only prove (\eqref{BiHomqua11}) and (\eqref{BiHomqua17}) and leave the rest to the reader.
 We define the following operations $x\succ _{(\alpha , \beta )}y:=x\nearrow _{(\alpha , \beta )}y+
x\searrow _{(\alpha , \beta )}y$,
$x\prec _{(\alpha , \beta )}y:=x\nwarrow _{(\alpha , \beta )}y+ x\swarrow _{(\alpha , \beta )}y$,
$x\vee _{(\alpha , \beta )}y:=x\searrow _{(\alpha , \beta )}y+ x\swarrow _{(\alpha , \beta )}y$,
$x\wedge _{(\alpha , \beta )}y:=x\nearrow _{(\alpha , \beta )}y+ x\nwarrow _{(\alpha , \beta )}y$
and $x\ast _{(\alpha , \beta )}y:=x\searrow _{(\alpha , \beta )}y+ x\nearrow _{(\alpha , \beta )}y+
x\swarrow _{(\alpha , \beta )}y+ x\nwarrow _{(\alpha , \beta )}y$, for all $x, y\in A$.
It is easy to get $x\succ _{(\alpha , \beta )}
y= \alpha (x)\succ \beta (y), x\prec _{(\alpha , \beta )} y= \alpha
(x)\prec \beta (y), x\vee _{(\alpha , \beta )} y= \alpha (x)\vee
\beta (y), x\wedge _{(\alpha , \beta )} y= \alpha (x)\wedge \beta
(y)$ and $x\ast _{(\alpha , \beta )} y= \alpha (x)\ast \beta (y)$
for all $x, y \in A$. By using the fact that $\alpha$ and $\beta $
are two commuting quadri-alternative algebra morphisms, one can compute, for
all $x, y, z \in A$:
\begin{align*}
    & \{\b(x),\a(y),z\}_{\a,\b}^{ne}+ \{\b(y),\a(x),z\}_{\a,\b}^e\\
    &=(\b(x) \wedge_{(\a,\b)} \a(y)) \ne_{(\a,\b)} \b(z)-\a\b(x) \ne_{(\a,\b)}(\a(y) \gr_{(\a,\b)} z)\\
    &+ (\b(y) \vee_{(\a,\b)} \a(x)) \ne_{(\a,\b)} \b(z) -\a\b(y) \se_{(\a,\b)}(\a(x) \ne_{(\a,\b)} z) \\
       &=(\a^2\b(x) \wedge \a^2\b(y)) \ne \b^2(z)-\a^2\b(x) \ne(\a^2\b(y) \gr \b^2(z)) \\
      &+ (\a^2\b(y) \vee \a^2\b(x)) \ne \b^2(z) -\a^2\b(y) \se (\a^2\b(x) \ne \b^2(z))   \\
   &= \{\a^2\b(x),\a^2\b(y),\b^2(z)\}^{ne}+ \{\a^2\b(y),\a^2\b(x),\b^2(z)\}^e=0,
\end{align*}
and
\begin{align*}
    &  \{x,\b(y),\a(z)\}_{\a,\b}^m +  \{x,\b(y),\a(z)\}_{\a,\b}^l \\
    & =(x \se_{(\a,\b)} \b(y)) \nw_{(\a,\b)} \a\b(z)- \a(x) \se_{(\a,\b)} (\b(y) \nw_{(\a,\b)} \a(z)) \\
    &+ (x \ast_{(\a,\b)} \b(z)) \se_{(\a,\b)} \a\b(y) - \a(x) \se_{(\a,\b)} (\b(z) \se_{(\a,\b)} \a(y))\\
    &= (\a^2(x) \se \a\b^2(y)) \nw \a\b^2(z)- \a^2(x) \se (\a\b^2(y) \nw \a\b^2(z)) \\
    &+ (\a^2(x) \ast \a\b^2(z)) \se \a\b^2(y) - \a^2(x) \se (\a\b^2(z) \se \a\b^2(y)) \\
    &=  \{\a^2(x),\a\b^2(y),\a\b^2(z)\}^m +  \{\a^2(x),\a\b^2(z),\a\b^2(y)\}^l =0.
\end{align*}

 \end{proof}

\begin{rem}
 Let $(A, \nwarrow, \swarrow, \nearrow, \searrow, \alpha , \beta )$ be a  BiHom-alternative quadri-algebra and $\widetilde{\a}, \widetilde{\b}: A \to A$ be two  BiHom-alternative quadri-algebra morphisms such that any of the maps $\a, \b, \widetilde{\a}, \widetilde{\b}$ commute. Define new multiplications on $A$ by:
 \begin{align*}
    & x \nearrow' y= \widetilde{\a}(x) \nearrow \widetilde{\b}(y), \quad
    x \searrow' y= \widetilde{\a}(x) \searrow \widetilde{\b}(y), \\
   & x \nwarrow' y= \widetilde{\a}(x) \nwarrow \widetilde{\b}(y), \quad
   x \swarrow' y= \widetilde{\a}(x) \swarrow \widetilde{\b}(y) .
 \end{align*}
 Then, one can prove that $(A',\nearrow',\searrow',\swarrow',\nwarrow',\a \circ \widetilde{\a}, \b \circ \widetilde{\b})$ is a BiHom-alternative quadri-algebra.
\end{rem}

%%%%%%%%%%%%%%%%%%%%%%%%%%%%%%%%%
%O-operators
\begin{df}
Let $(A,\pt,\gr,\a,\b)$ be a BiHom-pre-alternative algebra and $(V,L_{\pt},R_{\pt},L_{\gr},R_{\gr},\f,\p)$  be a bimodule. A linear map $T: V \to A$ is called an $\mathcal{O}$-operator of  $(A,\pt,\gr,\a,\b)$ associated to  $(V,L_{\pt},R_{\pt},L_{\gr},R_{\gr},\f,\p)$  if $T$ satisfies:  $T \f=\a T$, $T \p=\b T$ and
\begin{align}\label{O-op prealt}
& T(u) \gr T(v)=T(L_{\gr}(T(u))v+R_{\gr}(T(v))u),\ \   T(u) \pt T(v)=T(L_{\pt}(T(u))v+R_{\pt}(T(v))u),
\end{align}
for all $u, v \in V$.
\end{df}

The have the following results.
\begin{prop}
Let   $(V,L_{\pt},R_{\pt},L_{\gr},R_{\gr},\f,\p)$ be a bimodule of a BiHom-pre-alternative algebra  $(A,\pt,\gr,\a,\b)$ and $(A,\circ, \a,\b)$ be the associated BiHom-alternative algebra.
If $T$ is  an $\mathcal{O}$-operator of  $(A,\pt,\gr,\a,\b)$ associated to  $(V,L_{\pt},R_{\pt},L_{\gr},R_{\gr},\f,\p)$,  then $T$ is an $\mathcal{O}$-operator of $(A,\circ,\a,\b)$ associated to $(V,L_{\pt}+L_{\gr}, R_{\pt}+R_{\gr},\f,\p)$.
\end{prop}

\begin{prop}\label{last prop}
Let $(A,\pt,\gr,\a,\b)$ be a BiHom-pre-alternative algebra and  $(V,L_{\pt},R_{\pt},L_{\gr},R_{\gr},\f,\p)$  be a bimodule. Let  $T$ be  an $\mathcal{O}$-operator of  $(A,\pt,\gr,\a,\b)$ associated to  $(V,L_{\pt},R_{\pt},L_{\gr},R_{\gr},\f,\p)$. Then there exists a BiHom-alternative quadri-algebra structure on $V$ given for any $u,v \in V$ by
\begin{align}
& u \se v=L_{\gr}(T(u))v,\ u \ne v=R_{\gr}(T(v))u,\ u \sw v= L_{\pt}(T(u))v,\ u \nw v=R_{\pt}(T(v))u.
\end{align}
Therefore there exists a BiHom-pre-alternative algebra structure on $V$ given by
\begin{align*}
    & u \pt v=  u \sw v+ \ u \nw v,\ \ u \gr v =u \se v+ u \ne v,
\end{align*}
and $T$ is a homomorphism of BiHom-pre-alternative algebras.

Furthermore, $T(V)=\{T(v),\ v \in V\} \subset A$ is a BiHom-pre-alternative subalgebra of $A$ and there exists an induced BiHom-alternative quadri-algebra structure on $T(V)$ given by
\begin{align}\label{cor}
    & T(u) \se T(v)= T(u \se v), \  T(u) \ne T(v)= T(u \ne v), \nonumber \\
    &  T(u) \sw T(v)= T(u \sw v),\  T(u) \nw T(v)= T(u \nw v) .
\end{align}
Moreover, its corresponding associated horizontal BiHom-pre-alternative algebra structure on $T(V )$ is just the BiHom-pre-alternative subalgebra structure of $(A,\pt,\gr,\a,\b)$ and $T$ a  BiHom-alternative quadri-algebra homomorphism.
\end{prop}

\begin{proof}
Set $L=L_{\pt}+L_{\gr}$ and $R=R_{\pt}+R_{\gr}$.  For any $u,v,w \in V$, we have
\begin{align*}
    &(\p(u) \ast \f(u))\se \p(v)-\f\p(u) \se (\f(u) \se v) \\
    &=(L(T(\p(u)))\f(u)+R(T(\f(u)))\p(u))\se \p(v)-\f\p(u)\se (L_{\gr}(T(\f(u)))v) \\
    &=L_{\gr}(T(L(T(\p(u)))\f(u)+R(T(\f(u)))\p(u)))\p(v)-L_{\gr}(T(\f\p(u)))L_{\gr}(T(\f(u)))v \\
    & =L_{\gr}(T(\p(u)) \circ T(\f(u))) \p(v)-L_{\gr}(T(\f\p(u)))L_{\gr}(T(\f(u)))v =0.
\end{align*}
Furthermore,
\begin{align*}
 & (u \sw \p(v)) \nw \f\p(w)-\f(u) \sw (\p(v) \nw \f(w)+ \p(v) \ne \f(w)) \\
& =R_{\pt}(T(\f\p(w)))L_{\pt}(T(u))\p(v)-L_{\pt}(T(\f(u)))(R_{\pt}(T(\f(w)))\p(v)+R_{\gr}(T(\f(w)))\p(v))\\
  &  =R_{\pt}(T(\f\p(w)))L_{\pt}(T(u))\p(v)-L_{\pt}(T(\f(u)))R(T(\f(w)))\p(v),
\end{align*}
and
\begin{align*}
& (u \nw \p(w)+u \sw \p(w) ) \sw \f\p(v)-\f(u) \sw (\p(w) \se \f(v)+ \p(w) \sw \f(v))\\
&= L_{\pt}(T(R_{\pt}(\p(w))u+L_{\pt}(u)\p(w)))\f\p(v)-
L_{\pt}(T(\f(u)))(L_{\gr}(T(\p(w)))\f(v)+L_{\pt}(T(\p(w)))\f(v)) \\
&=L_{\pt}(T(u) \pt T(\p(w)))\f\p(v)-L_{\pt}(T(\f(u)))L(T(\p(w)))\f(v).
\end{align*}
This means that
\begin{align*}
 & \{u,\p(v),\f(w)\}^w_{\f,\p}+     \{u,\p(w),\f(v)\}^{sw}_{\f,\p} =\\
 &R_{\pt}(T(\f\p(w)))L_{\pt}(T(u))\p(v)-L_{\pt}(T(\f(u)))R(T(\f(w)))\p(v)\\
 & +L_{\pt}(T(u) \pt T(\p(w)))\f\p(v)-L_{\pt}(T(\f(u)))L(T(\p(w)))\f(v) =0,
\end{align*}
since $(V,L_{\pt},R_{\pt},L_{\gr},R_{\gr},\f,\p)$ is a bimodule of $(A,\pt,\gr,\a,\b)$.
The rest of identities can be proved using  analogous  computations, so they will be left to the reader.

\end{proof}

\begin{cor}
Let $(A,\pt,\gr,\a,\b)$ be a BiHom-pre-alternative algebra. Then there
exists a compatible BiHom-alternative quadri-algebra structure on $(A,\pt,\gr,\a,\b)$
such that $(A,\pt,\gr,\a,\b)$ is the associated horizontal
BiHom-pre-alternative algebra if and only if there exists an invertible
$\mathcal{O}$-operator of $(A,\pt,\gr,\a,\b)$.
%associated to certain bimodule
%$(l_\succ,r_\succ,l_\prec,r_\prec, V$) (hence $\dim V=\dim
%A$)
\end{cor}

 \begin{proof}
If there exists an invertible $\mathcal{O}$-operator
$T$ of $(A,\pt,\gr,\a,\b)$ associated to a bimodule

$(V,L_{\pt},R_{\pt},L_{\gr},R_{\gr},\f,\p)$, then by Proposition \ref{last prop}, there
exists a BiHom-alternative quadri-algebra structure on $V$.
Therefore we can define a BiHom-alternative quadri-algebra structure on $A$ by
equation \eqref{cor} such that $T$ is a BiHom-alternative quadri-algebra isomorphism,
that is,
$$x\searrow y=T(L_\succ(x)T^{-1}(y)),\;
x\nearrow y=T(R_\succ(y)T^{-1}(x)),\;$$$$ x\swarrow
y=T(L_\prec(x)T^{-1}(y)),\; x\nwarrow
y=T(R_\prec(y)T^{-1}(x)),\;\;\forall x,y\in A.$$ Moreover it is a
compatible BiHom-alternative quadri-algebra structure on $(A,\pt,\gr,\a,\b)$ since for
any $x,y\in A$, we have
$$x\succ y=T(T^{-1}(x)\succ T^{-1}(y))=T(R_\succ(y)T^{-1}(x)+R_\succ(x)T^{-1}(y))
=x\nearrow y+x\searrow y,$$
$$x\prec y=T(T^{-1}(x)\prec T^{-1}(y))=T(R_\prec(y)T^{-1}(x)+L_\prec(x)T^{-1}(y))
=x\nwarrow y+x\swarrow y.$$

Conversely, let $(A, \searrow, \nearrow, \nwarrow, \swarrow)$ be a
BiHom-alternative quadri-algebra and $(A,\pt,\gr,\a,\b)$ be the associated horizontal
BiHom-pre-alternative algebra. Then $(A,L_\searrow, R_\nearrow, L_\swarrow,
R_\nwarrow,\a,\b)$ is a bimodule of $(A,\pt,\gr,\a,\b)$ and the
identity map $id$ is an invertible $\mathcal{O}$-operator of $(A,\pt,\gr,\a,\b)$
associated to it.
 \end{proof}

%%%%%%%%%%%%%%%%%%%%%%%%%%%%%%%%%%%%%%%%%%%%%%%%%%%%%%%%%
Now, we introduce the following concept of Rota-Baxter operator on a BiHom-alternative quadri-algebra which is a particular case of $\mathcal{O}$-operator.
\begin{df}
Let $(A, \prec , \succ , \alpha , \beta )$ be a BiHom-pre-alternative
algebra. A Rota-Baxter operator of weight $0$ on $A$ is a linear map $R:
A\rightarrow A$ such that $R\alpha =\alpha  R$,
$R \beta =\beta  R$ and the following conditions are satisfied, for all $x, y\in A$:
\begin{eqnarray}
&& R(x)\succ R(y)=R(x\succ R(y)+ R(x)\succ y),  \label{baxter1}\\
&& R(x)\prec R(y)=R(x\prec R(y)+ R(x)\prec y).  \label{baxter2}
\end{eqnarray}
\end{df}
We know that $(A, \ast, \alpha , \beta )$ is a
BiHom-alternative algebra. Adding equations \eqref{baxter1} and
\eqref{baxter2}, one  obtains that $R$ is also a Rota-Baxter operator of
weight 0 for $(A, \ast)$:
\begin{eqnarray*}
&& R(x)\ast R(y)=R(x\ast R(y)+ R(x)\ast y).  \label{baxter3}
\end{eqnarray*}

Analogously to what happens for BiHom-quadri-algebras \cite{LiuMakhloufMeninPanaite}, Rota-Baxter operators allow different constructions for BiHom-alternative quadri-algebras.
\begin{cor} \label{operation}
Let $(A, \prec , \succ , \alpha , \beta )$ be a BiHom-pre-alternative
algebra and $R: A\rightarrow A$ be a Rota-Baxter operator of weight 0 for $A$.
Define new operations on $A$ by
\begin{eqnarray*}
x\searrow_R y=R(x)\succ y, ~ x\nearrow_R y=x\succ R(y), ~
x\swarrow_R y=R(x)\prec y ~and ~ x\nwarrow_R y=x\prec R(y).
\end{eqnarray*}
Then $(A, \nwarrow_R, \swarrow_R, \nearrow_R, \searrow_R, \alpha ,
\beta )$ is a BiHom-alternative quadri-algebra.
\end{cor}

\begin{lem}\label{pairRB}
Let $(A,\ast,\a,\b)$ be a BiHom-alternative algebra and $R, P$ two commuting Rota-Baxter operators on $A$ such that $R \alpha =\alpha  R$, $R \beta =\beta  R$, $P \alpha =\alpha  P$ and $P \beta =\beta  P$. Then $P$ is a Rota-Baxter operator on the BiHom-pre-alternative algebra $(A,\prec_{R},\succ_{R},\a,\b)$. where
$x \pt_R y=xR(y)$ and $x \gr_R y =R(x)y$.
\end{lem}
\begin{proof}
For all
$x, y\in A$, we have:
\begin{eqnarray*}
P(x)\succ_R P(y)&=& R(P(x))P(y)=P(R(x))P(y)\\
&=&P(R(x)P(y)+P(R(x))y)=P(R(x)P(y)+R(P(x))y)\\
&=& P(x\succ_R P(y)+ P(x)\succ_R y),
\end{eqnarray*}
and
\begin{eqnarray*}
P(x)\prec_R P(y)&=& P(x)R(P(y))=P(x)P(R(y))\\
&=&P(x P(R(y))+P(x)R(y))=P(x R(P(y))+P(x)R(y))\\
&=& P(x\prec_R P(y)+ P(x)\prec_R y).
\end{eqnarray*}
\end{proof}

\begin{cor}
In the setting of Lemma \ref{pairRB}, there exists
a BiHom-alternative quadri-algebra structure on the underlying vector
space $(A,\ast, \alpha , \beta)$, with operations defined by
\begin{eqnarray*}
&& x\searrow y=P(x)\succ_R y=P(R(x))y=R(P(x))y,\\
&& x\nearrow y=x\succ_R P(y)=R(x)P(y),\\
&& x\swarrow y=P(x)\prec_R y=P(x)R(y),\\
&& x\nwarrow y=x\prec_R P(y)=xR(P(y))=xP(R(y)).
\end{eqnarray*}
\end{cor}

\begin{proof}
We use the construction in Proposition \ref{operation} with the Rota-Baxter operator $P$  of weight $0$  and the BiHom-pre-alternative  algebra $(A, \prec_R, \succ_R,
\alpha , \beta)$.
\end{proof}

%%%%%%%%%%%%%%%%%%%%%%%%%%%%%%%%%%%%%%%%%%%%%%%%%%%%%%%%%%%%%%%%%%%%%%%%%%%%%%%%%%%%%%%%%%%%%%%%%%%%%%%%%%%%%%%%%%%%%%%%%%%%%%%%%%%%%%%%%%

%%%%%%%%%%%%%%%%%%%%%%%%%%%%%%%%%%%%%%%%%%%%%%%%%%%%%%%%%%%%%%%%%%%%%%%%%%%%%%%%%%%%%%%%%%%%%%%%%%%%%%%%%%%%%%%%%%%%%%%%%%%%%%%%%%%%%%%%%%

\end{document}